\documentclass{article}

\usepackage{amsmath}
\usepackage{amssymb}
\usepackage{amsthm}
\usepackage{enumerate}

\newtheorem{thm}{Theorem}

\theoremstyle{definition}
\newtheorem{dfn}[thm]{Definiton}

\theoremstyle{remark}
\newtheorem*{rem}{Remark}

\newcommand{\set}[1]{\{#1\}}
\newcommand{\Bigset}[1]{\Bigl\{#1\Bigr\}}

\newcommand{\bigprn}[1]{\bigl(#1\bigr)}
\newcommand{\Bigprn}[1]{\Bigl(#1\Bigr)}
\newcommand{\de}{\delta}
\newcommand{\De}{\Delta}
\newcommand{\vp}{\varphi}
\newcommand{\e}{\varepsilon}
\newcommand{\N}{\mathbb{N}}
\newcommand{\Z}{\mathbb{Z}}
\newcommand{\R}{\mathbb{R}}

\newcommand{\PSI}{\mathbf{\Psi}}

\begin{document}

\title{A fixed point theorem for contractive mappings that characterizes metric completeness}

\author{Mortaza Abtahi
\footnote{Address: \textit{School of Mathematics and Computer Sciences,
Damghan University, Damghan, P.O.BOX 36715-364, Iran}. Email: \texttt{abtahi@du.ac.ir}}}

\maketitle

\begin{abstract}
Inspired by the work of Suzuki in [Proc. Amer. Math. Soc. \textbf{136} (2008), 1861--1869]
we prove a fixed point theorem for contractive mappings
that generalizes a theorem of Geraghty in [Proc. Amer. Math. Soc., \textbf{40} (1973), 604--608]
and characterizes metric completeness.%
\end{abstract}

\section{Introduction}

Throughout this paper, we denote by $\N$ the set of all positive integers,
by $\Z^+$ the set of nonnegative integers, and
by $\R^+$ the set of nonnegative real numbers.
Given a set $X$ and a mapping $T:X\to X$, the $n$th iterate of
$T$ is denoted by $T^n$ so that $T^2x=T(Tx)$, $T^3x=T(T^2x)$ and so on.
A point $x\in X$ is called a \emph{fixed point} of $T$ if $T(x)=x$.

Let $(X,d\,)$ be a metric space. A mapping $T:X\to X$ is called  a \emph{contraction}
if there is $r\in[0,1)$ such that
\begin{equation}\label{eqn:Banach-contraction}
  \forall\, x,y\in X, \quad d(Tx,Ty)\leq rd(x,y).
\end{equation}
The following famous theorem is referred to as the Banach contraction
principle.

\begin{thm}[Banach, \cite{Banach}]
\label{thm:Banach}
  If $(X, d\,)$ is a complete metric space, then every contraction $T$ on $X$
  has a unique fixed point.
\end{thm}

The Banach fixed point theorem is very simple and powerful. It became
a classical tool in nonlinear analysis with many generalizations;
see
\cite{Boyd-Wong,
Caristi-1976,
Ciric-1974,
Ekeland-1974,
Kirk-2003,
Meir-Keeler,
Nadler-1969,
Subrahmanyam-1974,
Suzuki-2001,
Suzuki-2004,
Suzuki-2005,
Suzuki-2008,
Suzuki-NA-2009}.
For instance, the following result due to Boyd and Wong
is a great generalization of Theorem \ref{thm:Banach}.

\begin{thm}[{Boyd and Wong, \cite{Boyd-Wong}}]
\label{thm:Boyd-Wong}
  Let $(X,d\,)$ be a complete metric space, and let $T$
  be a mapping on $X$. Assume there exists a right-continuous function
  $\vp:\R^+\to\R^+$ such that $\vp(s)<s$ for $s>0$, and
  \begin{equation}\label{eqn:Boyd-Wong}
    \forall\, x,y\in X,\quad d(Tx,Ty) \leq \vp(d(x,y)).
  \end{equation}
  Then $T$ has a unique fixed point.
\end{thm}

There is an example of an incomplete metric space $X$ on which
every contraction has a fixed point, \cite{Connell-1959}.
This means that Theorem \ref{thm:Banach}
cannot characterize the metric completeness of $X$.
Recently, Suzuki in \cite{Suzuki-2008} proved the following
remarkable generalization of the classical Banach contraction theorem
that characterizes the metric completeness of $X$.

\begin{thm}[Suzuki, \cite{Suzuki-2008}]
\label{thm:Suzuki-2008}
  Define a function $\theta:[0, 1) \to (1/2, 1]$ by
  \begin{equation*}
    \theta(r)=
      \left\{
        \begin{array}{ll}
          1, & \hbox{if $0\leq r \leq (\sqrt5-1)/2$;} \\
          (1-r)r^{-2}, & \hbox{if $(\sqrt5-1)/2 \leq r \leq 1/\sqrt2$;} \\
          (1+r)^{-1}, & \hbox{if $1/\sqrt2 \leq r <1$.}
        \end{array}
      \right.
  \end{equation*}
  Let $(X,d\,)$ be a metric space. Then $X$ is complete if and only if
  every mapping $T$ on $X$ satisfying the following has a fixed point:
    \begin{itemize}
      \item  There exists $r\in[0, 1)$ such that
             \begin{equation}\label{eqn:suzuki-condition}
               \forall\, x,y\in X\
               \bigprn{\theta(r)d(x,Tx) \leq d(x,y) \ \Longrightarrow \
               d(Tx,Ty) \leq r d(x,y)}.
             \end{equation}
    \end{itemize}
\end{thm}

The above Suzuki's generalized version of Banach fixed point theorem
initiated a lot of work in this direction and
led to some important contribution in metric fixed point
theory. Several authors obtained
variations and refinements of Suzuki's result; see
\cite{Enjouji-Nakanishi-Suzuki-2009,
Kikkawa-Suzuki-2008,
Kikkawa-Suzuki-NA-2008,
Popescu-2009,
Singh-Mishra-2010,
Singh-Pathak-Mishra-2010}.

For a metric space $(X,d\,)$, a mapping $T:X\to X$ is called \emph{contractive} if
$d(Tx,Ty)<d(x,y)$, for all $x,y\in X$ with $x\neq y$.
Edelstein in \cite{Edelstein-62} proved that, on compact
metric spaces, every contractive mapping possesses a unique fixed
point theorem. Then in \cite{Suzuki-NA-2009} Suzuki generalized Edelstein's
result as follows. 

\begin{thm}[Suzuki, \cite{Suzuki-NA-2009}]
\label{thm:Suzuki-2009-NA}
  Let $(X,d\,)$ be a compact metric space and let $T$ be a mapping on X.
  Assume that
  \begin{equation}\label{eqn:suzuki-NA-2009}
    \forall\, x,y\in X\
    \Bigprn{\frac12 d(x,Tx) < d(x,y)\ \Longrightarrow\
    d(Tx,Ty)<d(x,y)}.
  \end{equation}
  Then $T$ has a unique fixed point.
\end{thm}

\noindent
It is interesting to note that,
although the above Suzuki's theorem generalizes Edelstein's theorem
in \cite{Edelstein-62}, these two theorems, as Suzuki mentioned in
\cite{Suzuki-NA-2009}, are not of the same type.

Let $T$ be a contractive mapping on a metric space $X$.
Choose a point $x\in X$ and set $x_n=T^nx$, for $n\in\N$. Criteria
for the sequence of iterates $\set{x_n}$ to be Cauchy are of interest,
for if it is Cauchy then it converges to a unique fixed point
of $T$, \cite{Geraghty-73}. Many papers have presented such criteria,
especially since the important paper of Rakotch \cite{Rakotch-62}.
For example, Geraghty in \cite{Geraghty-73} proved the following
theorem that gives a necessary and sufficient condition for a sequence
of iterates to be convergent.

\begin{thm}[Geraghty, \cite{Geraghty-73}]
\label{thm:Geraghty}
  Let $X$ be a complete metric space and let $T$ be a contractive mapping
  on $X$. Let $x\in X$ and set $x_n=T^n x$, $n\in\N$. Then
  $x_n$ converges to a unique fixed point of $T$ if and only if
  for any two subsequences $\set{x_{p_n}}$ and $\set{x_{q_n}}$, with
  $x_{p_n}\neq x_{q_n}$, if $\De_n\to1$ then $\de_n\to0$, where
  \[
    \de_n=d(x_{p_n},x_{q_n}), \qquad
    \De_n=d(Tx_{p_n},Tx_{q_n})/\de_n.
  \]
\end{thm}

Motivated by the works of Suzuki in \cite{Suzuki-2008} and \cite{Suzuki-NA-2009},
we prove a fixed point theorem for contractive mappings based on
Theorem \ref{thm:Geraghty} that characterizes metric completeness.

\section{Fixed Point Theorem}

Let $(X,d\,)$ be a metric space. We shall use the following notation:
for any pair of subsequences $\set{x_{p_n}}$
and $\set{x_{q_n}}$ of a given sequence $\set{x_n}$
in $X$, we let $\de_n=d(x_{p_n},x_{q_n})$ and
\[
  \De_n=
  \left\{\!\!
    \begin{array}{ll}
      0, & \hbox{$\de_n=0$;} \\
      d(Tx_{p_n},Tx_{q_n})/\de_n, & \hbox{$\de_n>0$.}
    \end{array}
  \right.
\]

\begin{thm}\label{thm:generalized-contractive-map}
  Let $(X,d)$ be a metric space and let a mapping $T:X\to X$ satisfy
  the following condition:
  \begin{equation}\label{eqn:generalized-contractive-map}
    \forall\, x,y\in X\
    \Bigprn{x\neq y,\ d(x,Tx) \leq d(x,y)\ \Longrightarrow\
    d(Tx,Ty) < d(x,y)}.
  \end{equation}

  \noindent
  Given $x\in X$, the following statements for the sequence $x_n=T^n x$, $n\in\N$,
  are equivalent:
  \begin{enumerate}[\upshape(i)]
    \item\label{item:1-pre-main}
    $\set{x_n}$ is a Cauchy sequence.

    \item\label{item:2-pre-main}
    For any two subsequences $\set{x_{p_n}}$ and $\set{x_{q_n}}$, with
    $d(x_{p_n},Tx_{p_n}) \leq d(x_{p_n},x_{q_n})$ for all $n$, if
    $\De_n\to1$ then $\de_n\to0$.
  \end{enumerate}
\end{thm}

\begin{proof}
  The implication \eqref{item:1-pre-main} $\Rightarrow$ \eqref{item:2-pre-main} is
  clear because $\set{x_n}$ is a Cauchy sequence and thus for any two subsequences
  $\set{x_{p_n}}$ and $\set{x_{q_n}}$ we always have $\de_n\to0$.

  We now prove \eqref{item:2-pre-main} $\Rightarrow$ \eqref{item:1-pre-main}.
  First, assume that $x_m=Tx_m$, for some $m$. Then $x_n=x_m$, for
  $n\geq m$, and particularly $\set{x_n}$ is a Cauchy sequence.
  Next, assume that $x_n\neq x_{n+1}$ for all $n$.
  Since $d(x_n,Tx_n) \leq d(x_n,Tx_n)$,
  condition \eqref{eqn:generalized-contractive-map} implies
  that the sequence $\de_n=d(x_n,x_{n+1})$ is strictly decreasing.
  Thus $\de_n\to\de$ for some nonnegative number $\de$.
  If $\de>0$, take $p_n=n$ and $q_n=n+1$.
  Then $d(x_{p_n},Tx_{p_n}) \leq d(x_{p_n},x_{q_n})$, for all  $n$, and
  $\De_n\to1$ while $\de_n\to\de\neq0$. This is a contradiction and hence
  $d(x_n,x_{n+1})\to0$.

  For every $n\in\N$, choose $k_n\in\N$ such that
  $d(x_m,x_{m+1})<1/n$ for $m\geq k_n$. If $\set{x_n}$ is not a Cauchy sequence,
  there exist $\e>0$ and sequences $\set{p_n}$ and $\set{q_n}$ of positive integers
  such that $q_n>p_n\geq k_n$ and $d(x_{p_n},x_{q_n})\geq\e$. We also
  assume that $q_n$ is the least such integer so that
  $d(x_{p_n},x_{q_n-1})<\e$. Therefore,
  \[
    \e \leq d(x_{p_n},x_{q_n})
       \leq d(x_{p_n},x_{q_n-1})+d(x_{q_n-1},x_{q_n}) < \e+1/n.
  \]

  \noindent
  This shows that $\de_n\to\e$. Since we have, for every  $n\in\N$,
  \[
    d(x_{p_n},Tx_{p_n}) \leq d(x_{p_n},x_{q_n})
     < d(x_{p_n},x_{q_n}),
  \]
  condition \eqref{eqn:generalized-contractive-map} shows that
  $d(Tx_{p_n},Tx_{q_n})<\de_n$. So
  \[
    \frac{\de_n-2/n}{\de_n} \leq
    \frac{d(Tx_{p_n},Tx_{q_n})}{\de_n} =\De_n<1.
  \]
  It shows that $\De_n\to1$ and thus $\de_n\to0$. This is a contradiction.
  Therefore, $\set{x_n}$ is a Cauchy sequence.
\end{proof}

The following is a Susuki-type generalization of Theorem \ref{thm:Geraghty}.

\begin{thm}\label{thm:Geraghty-1}
  Let $X$ be a complete metric space and let $T$ be a mapping on $X$ satisfying
  the following condition:
  \begin{equation}\label{eqn:suzuki-type-contractive-map}
    \forall\, x,y\in X\
    \Bigprn{\frac12d(x,Tx) < d(x,y)\ \Longrightarrow\
    d(Tx,Ty) < d(x,y)}.
  \end{equation}

  \noindent
  Given $x\in X$, the following statements for the sequence $x_n=T^n x$, $n\in\N$,
  are equivalent:
  \begin{enumerate}[\upshape(i)]
    \item\label{item:1-main}
    $x_n\to z$ in $X$, with $z$ a unique fixed point of $T$;

    \item\label{item:2-main}
    for any two subsequences $\set{x_{p_n}}$ and $\set{x_{q_n}}$, with
    $d(x_{p_n},Tx_{p_n}) \leq d(x_{p_n},x_{q_n})$ for all $n$, if
    $\De_n\to1$ then $\de_n\to0$.
  \end{enumerate}
\end{thm}

\begin{proof}
  First, let us prove that $T$ has at most one fixed point. If $z$ is
  a fixed point of $T$ and $z\neq y$ then $(1/2)d(z,Tz)<d(z,y)$ and condition
  \eqref{eqn:suzuki-type-contractive-map} implies that
  $d(Tz,Ty)<d(z,y)$. Since $Tz=z$, we must have $Ty\neq y$, i.e., $y$ is not a fixed point of $T$.

  The implication \eqref{item:1-main} $\Rightarrow$ \eqref{item:2-main} is
  clear. We prove \eqref{item:2-main} $\Rightarrow$ \eqref{item:1-main}.
  By Theorem \ref{eqn:generalized-contractive-map}, the sequence
  $\set{x_n}$ is Cauchy and, since the metric space $X$ is complete,
  $x_n\to z$ for some $z\in X$.
  We show that $Tz=z$. First note that,
  \begin{equation}\label{eqn:either-or-I}
    \forall n\
    \bigprn{d(x_n,x_{n+1}) < 2d(x_n,z)
    \quad\text{or}\quad
    d(x_{n+1},x_{n+2}) < 2d(x_{n+1},z)}.
  \end{equation}

  \noindent
  In fact, if $2d(x_n,z) \leq d(x_n,x_{n+1})$ and
  $2d(x_{n+1},z) \leq d(x_{n+1},x_{n+2})$ hold, for some $n$,
  then
  \begin{align*}
    2d(x_n,x_{n+1})
     & \leq 2d(x_n,z) + 2d(x_{n+1},z) \\
     & \leq d(x_n,x_{n+1}) + d(x_{n+1},x_{n+2})\\
     & < d(x_n,x_{n+1}) + d(x_n,x_{n+1})=2d(x_n,x_{n+1}).
  \end{align*}
  This is absurd and thus \eqref{eqn:either-or-I} must hold.
  Now condition \eqref{eqn:suzuki-type-contractive-map} together
  with \eqref{eqn:either-or-I} imply that
  \begin{equation}\label{eqn:either-or-II}
  \begin{split}
    \forall n\
    \bigprn{d(x_{n+1},Tz) < d(x_n,z)
    \quad\text{or}\quad
    d(x_{n+2},Tz) < d(x_{n+1},z)}.
  \end{split}
  \end{equation}

  \noindent
  Since $x_n\to z$, condition \eqref{eqn:either-or-II} implies
  the existence of a subsequence of $\set{x_n}$ that converges
  to $Tz$. This shows that $Tz = z$.
\end{proof}

\begin{rem}
  In Theorem \ref{thm:Geraghty-1}, if we replace condition
  \eqref{eqn:suzuki-type-contractive-map} with condition
  \eqref{eqn:suzuki-type-contractive-map-rem} below,
  we are still able to prove the theorem except for the uniqueness of
  the fixed point;
  \begin{equation}\label{eqn:suzuki-type-contractive-map-rem}
    \forall\, x,y\in X\
    \Bigprn{\frac12d(x,Tx) \leq d(x,y)\ \Longrightarrow\
    d(Tx,Ty) \leq d(x,y)},
  \end{equation}
\end{rem}

We next prove that $1/2$ in condition \eqref{eqn:suzuki-type-contractive-map}
of Theorem \ref{thm:Geraghty-1} is the best constant.

\begin{thm}\label{thm:1/2 is the best constant in Geraghty-1}
  For every $\eta\in(1/2,\infty)$, there exist a complete metric space $(X,d\,)$
  and a mapping $T:X\to X$ with the following properties:
  \begin{enumerate}[\upshape(i)]
    \item \label{item:best-const-1}
    the mapping $T$ has no fixed point in $X$,

    \item \label{item:best-const-2}
     $\eta d(x,Tx) \leq d(x,y)$ implies $d(Tx,Ty) < d(x,y)$, for all $x,y\in X$,

    \item \label{item:best-const-3}
    condition \eqref{item:2-main} of Theorem \ref{thm:Geraghty-1}
    holds for any choice of initial point.
  \end{enumerate}
\end{thm}

\begin{proof}
  We use the same method as in \cite[Theorem 4]{Suzuki-NA-2009}.
  Take $\eta\in(1/2,\infty)$ and choose $r\in(1/\sqrt2,1)$ such that
  $(1+r)^{-1}<\eta$. For $n\in\Z^+$, let $u_n=(1-r)(-r)^n$,
  and then set $X=\set{0,1}\cup\set{u_n:n\in\Z^+}$.
  Define a mapping $T$ on $X$ by $T0=1$, $T1=u_0$ and $Tu_n=u_{n+1}$
  for $n\in\Z^+$. Obviously $T$ has no fixed point in $X$ and thus
  \eqref{item:best-const-1} is proved. We now prove
  part \eqref{item:best-const-2}.
  In \cite{Suzuki-2008}, Suzuki showed the following
  \[
    \forall\, x,y\in X\
    \bigprn{(1+r)^{-1}d(x,Tx) < d(x,y) \ \Longrightarrow \
    d(Tx,Ty) \leq rd(x,y)}.
  \]

  \noindent
  Now, if $\eta d(x,Tx) \leq d(x,y)$ then
  $(1+r)^{-1} d(x,Tx) < d(x,y)$
  and thus $d(Tx,Ty) \leq rd(x,y)<d(x,y)$. This proves part \eqref{item:best-const-2}.
  Finally, we show that, in this setting,
  condition \eqref{item:2-main} of  Theorem \ref{thm:Geraghty-1} holds.
  Take an arbitrary element $x\in X$ as initial point and
  set $x_n=T^nx$, $n\in\N$. Then $\set{x_n:n\geq2}$ is a
  subsequence of $\set{u_n}$ and since $u_n\to0$ the sequence
  $\set{x_n}$ is Cauchy. Hence if $\set{x_{p_n}}$ and $\set{x_{q_n}}$
  are two subsequences of $\set{x_n}$ we have
  $d(x_{p_n},x_{q_n})\to0$.
\end{proof}

Using a similar method as in \cite{Geraghty-73}, we can easily
convert the sequential condition \eqref{item:2-main}
in Theorem \ref{thm:Geraghty-1} to the more customary
functional form. Following \cite{Geraghty-73}, we define a class of test
functions as follows:

\begin{dfn}
  We say that $\psi$ is of class $\PSI$, written $\psi\in\PSI$, if
  $\psi$ is a function of $\R^+$ into $[0,1]$ and,
  for every sequence $\set{s_n}$ of positive numbers,
  the condition $\psi(s_n)\to1$ implies that $s_n\to0$.
\end{dfn}

We do not assume that $\psi$ is continuous in any sense.
We only require that if $\psi$ gets near one, it does so only near
zero, \cite{Geraghty-73}.

\begin{thm}\label{thm:Geraghty-2}
  Let $X$ be a complete metric space and let $T$ be a mapping on $X$
  satisfying \eqref{eqn:suzuki-type-contractive-map}. For any $x\in X$,
  the following statements, for the sequence $x_n=T^n x$, $n\in\N$,
  are equivalent:
  \begin{enumerate}[\upshape(i)]
    \item \label{item:1-PSI}
    $x_n\to z$ in $X$, with $z$ the unique fixed point of $T$;

    \item \label{item:2-PSI}
    there is $\psi\in\PSI$
    such that,  for all $m,n\in \N$,
  \begin{equation}\label{eqn:eta-psi-contraction}
    d(x_n,Tx_n)\leq d(x_n,x_m)\ \Longrightarrow\
    d(Tx_n,Tx_m)\leq \psi(d(x_n,x_m))d(x_n,x_m).
  \end{equation}
  \end{enumerate}
\end{thm}

\begin{proof}
  It suffices to show that condition \eqref{item:2-PSI} of the theorem is
  equivalent to condition \eqref{item:2-main} of
  Theorem \ref{thm:Geraghty-1}. First assume that $\psi\in\PSI$
  exists and satisfies \eqref{eqn:eta-psi-contraction}.
  Let $\set{x_{p_n}}$ and $\set{x_{q_n}}$ be subsequences of $\set{x_n}$
  such that $d(x_{p_n},Tx_{p_n})\leq d(x_{p_n},x_{q_n})$.
  Assume that $\De_n\to1$. Since
  $d(x_{p_n},Tx_{p_n})\leq d(x_{p_n},x_{q_n})$,
  condition \eqref{eqn:eta-psi-contraction} shows that
  $\psi(\de_n)\to1$. Since $\psi\in\PSI$, we have
  $\de_n\to0$.

  Next assume that the sequential condition \eqref{item:2-main}
  of Theorem \eqref{thm:Geraghty-1} holds.
  Define $\psi:\R^+\to[0,1]$ as follows:
  Given $s\in\R^+$, if there exist no $m,n\in\N$ for which
  $d(x_n,Tx_n) \leq s \leq d(x_n,x_m)$,
  define $\psi(s)=0$; otherwise define
  \[
    \psi(s) =\sup\Bigset{
      \frac{d(Tx_n,Tx_m)}{d(x_n,x_m)}: d(x_n,Tx_n) \leq s \leq d(x_n,x_m)}.
  \]

  \noindent
  Since $T$ satisfies condition \eqref{eqn:suzuki-type-contractive-map},
  we have $0\leq \psi(s)\leq 1$, for every $s$.
  Assume that $\psi(s_n)\to1$ for some sequence $\set{s_n}$ in $\R^+$.
  Take a sequence $\set{r_n}$ of positive
  numbers such that $r_n<\psi(s_n)$ and $r_n\to1$.
  Then, there exist two subsequences $\set{x_{p_n}}$ and $\set{x_{q_n}}$ for which
  $d(x_{p_n},Tx_{p_n}) \leq s_n \leq d(x_{p_n},x_{q_n})$ and
  \[
     r_n<\frac{d(Tx_{p_n},Tx_{q_n})}{d(x_{p_n},x_{q_n})}
       \leq \psi(s_n).
  \]

  \noindent
  Therefore, $\De_n\to1$ and condition \eqref{item:2-main} of Theorem
  \ref{thm:Geraghty-1} shows $\de_n\to0$.
\end{proof}

We now apply the above results to obtain a criterion for convergence of the
iteration from an arbitrary starting point.

\begin{thm}\label{thm:Geraghty-3}
  Let $T$ be a mapping on a complete metric space $X$.
  Assume that, for some $\psi\in\PSI$, we have
  \begin{equation}
    \forall\, x,y\in X\
    \Bigprn{\frac12d( x,Tx) < d(x,y)\ \Longrightarrow\
    d(Tx,Ty) < \psi(d(x,y))d(x,y)}.
  \end{equation}
  Then for any choice of initial point $x$, the iteration
  $x_n=T^n x$, $n\in\N$, converges to the unique fixed point $z$ of
  $T$ in $X$.
\end{thm}

\section{Metric Completion}

In this section, we discuss the metric completeness.

\begin{thm}
  Let $(X,d)$ be a metric space. Then $X$ is complete if and only if
  every mapping $T:X\to X$ satisfying the following two conditions
  has a fixed point in $X$;
  \begin{enumerate}[\upshape(i)]
    \item\label{item:1-complete}
     There exists a constant $\eta\in(0,1/2]$ such that
     $\eta d(x,Tx) < d(x,y)$ implies $d(Tx,Ty) < d(x,y)$, for all $x,y\in X$.

    \item\label{item:2-complete}
    There exists a point $x\in X$ such that condition \eqref{item:2-main}
    of Theorem \ref{thm:Geraghty-1} holds.
  \end{enumerate}
\end{thm}

\begin{thm}\label{thm:Completeness @ Geraghty}
    For a metric space $(X, d\,)$, the following are equivalent:
  \begin{enumerate}[\upshape(i)]
    \item \label{item:X is complete @ Geraghty}
    The space $X$ is complete.

    \item \label{item:Geraghty-condition}
    For any mapping $T$ on $X$ that satisfies \eqref{eqn:suzuki-type-contractive-map},
    conditions \eqref{item:1-main} and \eqref{item:2-main} of Theorem \ref{thm:Geraghty-1} are
    equivalent.
  \end{enumerate}
\end{thm}

\begin{proof}
  $\eqref{item:X is complete @ Geraghty}\Rightarrow\eqref{item:Geraghty-condition}$
  follows from Theorem \ref{thm:Geraghty-1}. To prove
  $\eqref{item:Geraghty-condition}\Rightarrow\eqref{item:X is complete @ Geraghty}$,
  towards a contradiction, let the metric space $X$ be incomplete.
  Then, as in the proof of Theorem 4 in \cite{Suzuki-2008}, there
  exists a Cauchy sequence $\set{u_n}$ which does not converge. Define a function
  $\rho:X\to\R^+$ by $\rho(x) = \lim_n d(x,u_n)$, for $x\in X$. Note that $\rho$
  is well-defined because $\set{d(x,u_n)}$ is a Cauchy sequence in $\R$, for every
  $x\in X$. The following are obvious:

  \begin{itemize}
    \item $\rho(x)-\rho(y)\leq d(x, y)\leq \rho(x) + \rho(y)$,  for $x, y\in X$,

    \item $\rho(x) > 0$ for all $x\in X$,

    \item $\rho(u_n)\to0$ as $n\to\infty$.
  \end{itemize}

  \noindent
  Define a mapping $T:X\to X$ as follows: For each $x\in X$, since $\rho(x) > 0$
  and $\rho(u_n)\to0$, there exists $m\in\N$ such that
  \begin{equation}\label{eqn:rho(un)<1/7rho(x)}
     \rho(u_n) < \frac{\rho(x)}7,
     \qquad (n\geq m).
  \end{equation}
  Put $T(x)=u_m$. In case $x=u_k$, for some $k$, we choose $m$ large enough such that
  $m>k$ and \eqref{eqn:rho(un)<1/7rho(x)} holds.
  It is obvious that $\rho(Tx)<\rho(x)/7$ so that
  $Tx\neq x$, for every $x\in X$. That is, $T$ does not have a fixed point.
  Let us prove that $T$ satisfies \eqref{eqn:suzuki-type-contractive-map}.
  Fix $x, y\in X$ with $(1/2)d(x, Tx) < d(x, y)$. In the case where
  $2 \rho(x)\leq\rho(y)$, we have
  \begin{align*}
    d(Tx,Ty)
     & \leq \rho(Tx) + \rho(Ty)
       < \frac13 \bigprn{\rho(x) + \rho(y)} \\
     & \leq \frac13 \bigprn{\rho(x) + \rho(y)} + \frac23 \bigprn{\rho(y) - 2 \rho(x)} \\
     & = \rho(y) - \rho(x) \leq d(x,y).
  \end{align*}

  \noindent
  In the other case, where $\rho(y)<2 \rho(x)$, we have
  \begin{align*}
    d(x,y) > \frac12 d(x,Tx) \geq \frac12\bigprn{\rho(x)-\rho(Tx)}
      \geq \frac12\Bigprn{1-\frac17}\rho(x)=\frac37\rho(x).
  \end{align*}

  \noindent
  Therefore,
  \begin{align*}
    d(Tx,Ty)
     & \leq \rho(Tx) + \rho(Ty) < \frac17\bigprn{\rho(x)+\rho(y)} \\
     & \leq \frac17\bigprn{\rho(x)+2\rho(x)} = \frac37\rho(x) \leq d(x,y).
  \end{align*}

  \noindent
  Finally, we show that, given $x\in X$, condition \eqref{item:2-main} of Theorem \ref{thm:Geraghty-1}
  holds for the iteration sequence $x_n=T^nx$, $n\in\N$.
  The definition of $T$ shows that there exists a sequence
  $\set{m_n}$ of positive integers such that $m_n<m_{n+1}$
  and $x_n=u_{m_n}$. Hence $\set{x_n}$ is a subsequence of
  $\set{u_n}$. Now, if $\set{x_{p_n}}$ and $\set{x_{q_n}}$ are subsequences
  of $\set{x_n}$, they are also subsequences of $\set{u_n}$ and thus
  $d(x_{p_n},x_{q_n})\to0$ because $\set{u_n}$ is a Cauchy sequence.
  This shows that condition \eqref{item:2-main} of Theorem \ref{thm:Geraghty-1}
  holds for the sequence $\set{x_n}$.
  This is a contradiction since condition
  \eqref{item:1-main} of Theorem \ref{thm:Geraghty-1}
  does not hold for the sequence $\set{x_n}$.
\end{proof}





\end{document}